\documentclass[11pt, reqno]{amsart}
\usepackage{amsfonts,latexsym,enumerate}
\usepackage{amsmath}
\usepackage{amscd}
\usepackage{float,amsmath,amssymb,mathrsfs,bm,multirow,graphics}
\usepackage[dvips]{graphicx}
\usepackage[percent]{overpic}
\usepackage{color}
\usepackage[numbers,sort&compress]{natbib}
\usepackage{epstopdf}
\usepackage{tikz,pgf}
\usetikzlibrary{decorations.markings, decorations.pathmorphing}
\usetikzlibrary{arrows,decorations,backgrounds,scopes,plotmarks}

\newcommand{\R}{{\Bbb R}}

\newcommand{\C}{{\Bbb C}}

\newcommand{\I}{{\bf I}}
\newcommand{\0}{{\bf 0}}

\newcommand{\tr}{\text{\upshape tr\,}}

\newcommand{\diag}{\text{\upshape diag\,}}
\newcommand{\re}{\text{\upshape Re\,}}
\newcommand{\im}{\text{\upshape Im\,}}

\newtheorem{theorem}{Theorem}
\newtheorem{proposition}{Proposition}[section]
\newtheorem{lemma}[proposition]{Lemma}

\newtheorem{remark}[proposition]{Remark}

\numberwithin{equation}{section}

\usepackage[colorlinks=true]{hyperref}
\hypersetup{urlcolor=blue, citecolor=red, linkcolor=blue}

\input epsf
\title[Asymptotics for the Sasa--Satsuma equation]
{Asymptotics for the Sasa--Satsuma equation \\ in terms of a modified Painlev\'e II transcendent}

\author{Lin Huang}
\address{School of sciences, Hangzhou Dianzi University, 310018, Hangzhou, China.}
\email{lin.huang@hdu.edu.cn}

\author{Jonatan Lenells}
\address{Department of Mathematics, KTH Royal Institute of Technology, 100 44 Stockholm, Sweden.}
\email{jlenells@kth.se}

\begin{document}

\begin{abstract}
\noindent
We consider the initial-value problem for the Sasa--Satsuma equation on the line with decaying initial data. Using a Riemann--Hilbert formulation and steepest descent arguments, we compute the long-time asymptotics of the solution in the sector $|x| \leq M t^{1/3}$, $M$ constant. It turns out that the asymptotics can be expressed in terms of the solution of a modified Painlev\'e II equation. Whereas the standard Painlev\'e II equation is related to a $2 \times 2$ matrix Riemann--Hilbert problem, this modified Painlev\'e II equation is related to a $3 \times 3$ matrix Riemann--Hilbert problem.
\end{abstract}

\maketitle




\section{Introduction}
In this paper, we consider the long-time behavior of the solution of the Sasa--Satsuma equation \cite{ss-1991}
\begin{align}\label{sseq}
  u_t-u_{xxx}-6|u|^2u_x-3u(|u|^2)_x=0,
\end{align}
with initial data $u(x,0)= u_0(x) \in \mathcal{S}(\R)$ in the Schwartz class. Our main result shows that $u(x,t)$ admits an expansion to all orders in the asymptotic sector $|x| < M t^{1/3}$ of the form
\begin{align}\label{uasymptoticsintro}
u(x,t) \sim \sum_{j=1}^\infty \frac{u_j(y)}{t^{j/3}}, \qquad t \to \infty,
\end{align}
where $\{u_j(y)\}_1^\infty$ are smooth functions of $y \doteq x/(3t)^{1/3}$ and $M > 0$ is a constant. It also shows that the leading coefficient $u_1(y)$ is given by
$$u_1(y) = i\frac{u_P(y)}{3^{1/3} \sqrt{2}},$$ 
where $u_P(y)$ satisfies the following modified Painlev\'e II equation:
\begin{align}\label{complexpainleveII}
  u_{P}''(y) + yu_P(y) + 2u_P(y)|u_P(y)|^2 = 0.
\end{align}  
Equation (\ref{complexpainleveII}) coincides with the standard Painlev\'e II equation
\begin{align}\label{painleveII}  
  u_{P}''(y) - yu_P(y) - 2u_P(y)^3 = 0,
\end{align}
except for a sign difference and the presence of the absolute value squared in the last term. We will show that (\ref{complexpainleveII}) is related to a $3 \times 3$ matrix RH problem much in the same way that (\ref{painleveII}) is related to a $2 \times 2$ matrix RH problem cf. \cite{FIKN2006}. 
In the case of a real-valued solution, equation (\ref{sseq}) reduces to a version of the mKdV equation, (\ref{complexpainleveII}) reduces (up to a sign) to (\ref{painleveII}), and the expansion (\ref{uasymptoticsintro}) reduces to the analogous asymptotic formula for the corresponding mKdV equation (see \cite{dz-1993}, and \cite{CL2019} for the higher order terms, in the case of the standard mKdV equation).

It turns out that the leading coefficient $u_1(y)$ in (\ref{uasymptoticsintro}) has constant phase, that is, $u_1(y) = |u_1(y)| e^{i\alpha}$ where $\alpha \in \R$ is independent of $y$. It is somewhat remarkable that this is the case for any choice of the complex-valued initial data $u_0(x) = u(x,0)$; however, we also recall that the Sasa--Satsuma has a class of one-soliton solutions of constant phase (see \cite{ss-1991} or \cite{ATW2011}):
$$u_{\text{1-sol}}(x,t) = \frac{\sqrt{2} a e^{a(x + a^2t - x_0)}e^{i\phi}}{1 + e^{2a(x + a^2 t - x_0)}}, \qquad \text{$a, \phi, x_0$ real constants}.$$

The starting point for our analysis is a Riemann--Hilbert (RH) representation for the solution of (\ref{sseq}) obtained via the inverse scattering transform formalism. The asymptotic formula (\ref{uasymptoticsintro}) is derived by performing a Deift--Zhou \cite{dz-1993} steepest descent analysis of this RH problem. The main novelty compared with the analogous derivation for the mKdV equation is that the Lax pair of (\ref{sseq}) involves $3\times 3$ instead of $2\times 2$ matrices. 

The inverse scattering problem for (\ref{sseq}) was studied already by Sasa and Satsuma \cite{ss-1991}. 
The initial-boundary value problem for (\ref{sseq}) on the half-line was considered in \cite{xf-2013}. Asymptotic formulas for the long-time behavior in the sector $0 < c_1 < x < c_2$ were obtained in \cite{LG2019,lgx-2018}. 

Our main results are presented in Section \ref{mainsec}. They are stated in the form of three theorems (Theorem \ref{th1}-\ref{th3}) whose proofs are given in Section \ref{th1proofsec}, \ref{th2proofsec}, and \ref{th3proofsec}, respectively. Section \ref{laxsec} recalls the Lax pair formulation of (\ref{sseq}). The RH problem associated with the modified Painlev\'e II equation (\ref{complexpainleveII}) is discussed in Appendix \ref{appA}. Appendix \ref{appB} considers an extension of this RH problem which is needed to obtain the higher order terms in (\ref{uasymptoticsintro}).

\section{Main results}\label{mainsec}
Our first theorem shows how solutions of (\ref{sseq}) can be constructed starting from an appropriate spectral function $\rho_1(k)$. We let $\mathcal{S}(\R)$ denote the Schwartz class of smooth (complex-valued) rapidly decaying functions.  

\begin{theorem}[Construction of solutions]\label{th1}
Suppose $\rho_1 \in \mathcal{S}(\R)$.
Define the $3 \times 3$-matrix valued jump matrix $v(x,t,k)$ by
\begin{align}\label{vdef}
v(x,t,k) = \begin{pmatrix} \I_{2\times 2} & \rho^{\dag}(\bar{k})e^{-2ikx+8ik^3t}\\\rho(k)e^{2ikx-8ik^3t}&1+\rho(k)\rho^{\dag}(\bar{k})\end{pmatrix},
\end{align}
where
$$\rho(k)\doteq \begin{pmatrix}\rho_1(k) & \rho_2(k)\end{pmatrix}, \quad \rho^{\dag}(\bar{k})\doteq \begin{pmatrix}\overline{\rho_1(\bar{k})}\\\overline{\rho_2(\bar{k})}\end{pmatrix}, \quad \rho_2(k) \doteq \overline{\rho_1(-\bar{k})}.$$

Then the $3 \times 3$-matrix RH problem
\begin{itemize}
\item $m(x,t,k)$ is analytic for $k \in \C\setminus \R$ and extends continuously to $\R$ from the upper and lower half-planes;
\item the boundary values $m_\pm(x,t,k) = m(x,t,k\pm i0)$ obey the jump condition $m_+(x,t,k)=m_-(x,t,k)v(x,t,k)$ for $k \in \R$;
\item $m(x,t,k)=I+O(k^{-1})$ as $k\to \infty$;
\end{itemize}
has a unique solution for each $(x,t) \in \R^2$ and the limit $\lim_{k\to \infty} (k m(x,t,k))_{13}$ exists for each $(x,t) \in \R^2$. Moreover, the function $u(x,t)$ defined by 
\begin{align}\label{ulim}
  u(x,t) = 2i\lim_{k\to \infty} \big(k m(x,t,k)\big)_{13}
\end{align}
is a smooth function of $(x,t) \in \R^2$ with rapid decay as $|x| \to \infty$ which satisfies the Sasa--Satsuma equation \eqref{sseq} for $(x,t) \in \R^2$. 
\end{theorem}
\begin{proof}
See Section \ref{th1proofsec}.
\end{proof}

Our second theorem gives the long-time asymptotics of the solutions constructed in Theorem \ref{th1} in the sector $|x| \leq M t^{1/3}$. 

\begin{theorem}[Asymptotics of constructed solutions]\label{th2}
Under the assumptions of Theorem \ref{th1}, the solution $u(x,t)$ of (\ref{sseq}) defined in (\ref{ulim}) satisfies the following asymptotic formula as $t \to \infty$:
\begin{align}\label{uasymptotics}
u(x,t) = \sum_{j=1}^N \frac{u_j(y)}{t^{j/3}} + O\big(t^{-\frac{N+1}{3}}\big),  \qquad |x| \leq M t^{1/3}, 
\end{align}
where 
\begin{itemize}
\item The formula holds uniformly with respect to $x$ in the given range for any fixed $M > 0$ and $N \geq 1$.
\item The variable $y$ is defined by
\begin{align*}
  y= \frac{x}{(3t)^{1/3}}.
\end{align*}

\item $\{u_j(y)\}_1^\infty$ are smooth functions of $y \in \R$.

\item The function $u_1(y)$ is given by 
\begin{align}\label{u1expression}
u_1(y) = i\frac{u_P(y;s)}{3^{1/3} \sqrt{2}}, 
\end{align}
where $s \doteq  \rho_1(0)$ and $u_P(y; s)$ denotes the smooth solution of the modified Painlev\'e II equation (\ref{complexpainleveII}) corresponding to $s$ according to Lemma \ref{painlevelemma}. In particular, $u_1(y)$ has a constant phase, that is, $\arg u_1$ is independent of $y$. 
\end{itemize}
\end{theorem}
\begin{proof}
See Section \ref{th2proofsec}.
\end{proof}

\begin{remark}[Hierarchy of differential equations]\upshape
Substituting the expansion (\ref{uasymptotics}) into (\ref{sseq}) and identifying coefficients of powers of $t^{-1/3}$, we infer that the coefficients $\{u_j(y)\}_1^\infty$ in (\ref{uasymptotics}) satisfy a hierarchy of linear ordinary differential equations. The first two equations in this hierarchy are 
\begin{subequations}
\begin{align}\label{ujhierarchya}
& u_1''' + y u_1' + u_1 = - 3^{5/3} (3 |u_1|^2u_1' + u_1^2 \bar{u}_1'),
	\\
& u_2''' + y u_2' + 2u_2 = - 3^{5/3}\big( 3|u_1|^2 u_2' + u_1^2 \bar{u}_2' + 3\bar{u}_1 u_1' u_2 + 2u_1 \bar{u}_1' u_2 + 3 u_1 u_1' \bar{u}_2 \big).
\end{align}	
\end{subequations}
As expected, the function $u_1(y)$ in (\ref{u1expression}) satisfies the first of these equations.
Indeed, if $u_1(y)$ is given by (\ref{u1expression}) where $u_P(y) \equiv u_P(y; s)$ satisfies (\ref{complexpainleveII}), then (\ref{ujhierarchya}) reduces to the equation $|u_P(y)|^3 (\arg u_P)'(y) = 0$, which is satisfied for solutions $u_P$ of constant phase.   
\end{remark}

By applying the above two theorems in the case when $\rho_1(k)$ is the ``reflection coefficient'' corresponding to some given initial data $u_0(x)$, we obtain our third theorem, which establishes the asymptotic behavior of the solution of the initial-value problem for (\ref{sseq}) in the sector $|x| \leq M t^{1/3}$. Before stating the theorem, we  introduce some notation. 

Given $u_0 \in \mathcal{S}(\R)$, define $\mathsf{U}_0(x)$ and $\Lambda$ by
\begin{align*}
\mathsf{U}_0(x) = \begin{pmatrix}
0 	& 0 & u_0(x) \\
0 & 0 & \overline{u_0(x)} \\
-\overline{u_0(x)} & -u_0(x) & 0
\end{pmatrix}, \quad
 \Lambda=\begin{pmatrix}
1 	& 0 & 0 \\
0 & 1 & 0 \\
 0 & 0 & -1
\end{pmatrix}.
\end{align*}
Define the $3 \times 3$-matrix valued function $X(x,k)$ as the unique solution of the Volterra integral equation
\begin{align*}
  X(x,k) = I - \int_x^{\infty} e^{i k(x'-x) \hat{\Lambda}} (\mathsf{U}_0X)(x',k) dx', \qquad x \in \R, \ k \in \R,
\end{align*}
where $\hat{\Lambda}$ acts on a matrix $A$ by $\hat{\Lambda} A = [\Lambda, A]$, i.e., $e^{\hat{\Lambda}} A = e^{\Lambda} Ae^{-\Lambda}$.
Define the scattering matrix $s(k)$ by 
\begin{align}\label{sdef}
& s(k) = I - \int_\R e^{ikx\hat{\Lambda}}(\mathsf{U}X)(x,k)dx, \qquad k \in \R.
\end{align}
Then the ``reflection coefficient'' $\rho_1(k)$ is defined by
\begin{align}\label{rho1def}
\rho_1(k) = \frac{\overline{s_{13}(k)}}{\overline{s_{33}(k)}}, \qquad k \in \R.
\end{align}
We will see in Section \ref{th3proofsec} that the $(33)$ entry $s_{33}(k)$ of $s(k)$ has an analytic continuation to the upper half-plane.
Possible zeros of $s_{33}(k)$ give rise to poles in the RH problem, see (\ref{mPsidef}). 
For simplicity, we assume that no such poles are present (solitonless case).

\begin{theorem}[Asymptotics for initial value problem]\label{th3}
Suppose $u_0 \in \mathcal{S}(\R)$ and define $s(k)$ and $\rho_1(k)$ by (\ref{sdef}) and (\ref{rho1def}). Suppose the (33)-entry $s_{33}(k)$ is nonzero for $\im k \geq 0$.

Then $\rho_1 \in \mathcal{S}(\R)$ and the solution $u(x,t)$ of (\ref{sseq}) defined in terms of $\rho_1(k)$ by (\ref{ulim}) is the unique solution of the initial value problem for (\ref{sseq}) with initial data $u(x,0) = u_0(x)$ and rapid decay as $|x| \to \infty$. Moreover, $u(x,t)$ obeys the asymptotic formula (\ref{uasymptotics}) as $t \to \infty$.
\end{theorem}
\begin{proof}
See Section \ref{th3proofsec}.
\end{proof}

\begin{remark}[Scattering transform]\upshape
Let $S$ denote the subset of $\mathcal{S}(\R)$ consisting of all functions $u_0(x)$ such that the associated scattering matrix $s(k)$ defined in (\ref{sdef}) satisfies $s_{33}(k) \neq 0$ for $\im k \geq 0$.
Theorem \ref{th3} shows that the map which takes $u_0(x)$ to $\rho_1(k)$ (the scattering transform) is a bijection from $S$ onto its image in $\mathcal{S}(\R)$. The inverse of this map (the inverse scattering transform) is given by the construction of Theorem \ref{th1} for $t = 0$.
\end{remark}

\section{Lax pair}\label{laxsec}
An essential ingredient in the proofs of Theorem \ref{th1}-\ref{th3} is the fact that equation \eqref{sseq} is the compatibility condition of the Lax pair equations \cite{ss-1991}
\begin{align}\label{psilax}
\begin{cases}
\psi_x(x,t,k) =L(x,t,k)\psi(x,t,k),\\
\psi_t(x,t,k) =Z(x,t,k)\psi(x,t,k),
\end{cases}
\end{align}
where $k \in \C$ is the spectral parameter, $\psi(x,t, k)$ is a $3\times 3$-matrix valued eigenfunction, the $3\times 3$-matrix valued functions $L$ and $Z$ are defined by
\begin{equation}\label{philax}
L(x,t,k)=\mathcal{L}(k) +\mathsf{U}(x,t), \quad Z(x,t,k)=\mathcal{Z}(k)+\mathsf{V}(x,t,k)
\end{equation}
where $\mathcal{L}(k) = -i k \Lambda$, $\mathcal{Z}(k) = 4 i k^3 \Lambda$,
\begin{align}\label{L1def}
& \Lambda=\begin{pmatrix}
1 	& 0 & 0 \\
0 & 1 & 0 \\
 0 & 0 & -1
\end{pmatrix},
 \qquad \mathsf{U} = \begin{pmatrix}
0 	& 0 & u \\
0 & 0 & \bar u \\
-\bar u & -u & 0
\end{pmatrix},
	\\ \nonumber
& \mathsf{V} =k^2 \mathsf{V}^{(2)}+k \mathsf{V}^{(1)}+\mathsf{V}^{(0)},
	\\ \nonumber
& \mathsf{V}^{(2)}=-4 \mathsf{U}, \quad  
\mathsf{V}^{(1)}=-2i \begin{pmatrix}|u|^2&u^2&u_x\\\bar{u}^2&|u|^2&\bar u_x \\\bar u_x&u_x&-2|u|^2\end{pmatrix},
	\\ \label{Z1def}
&  \mathsf{V}^{(0)}= 4|u|^2 \mathsf{U}+\mathsf{U}_{xx}-(u\bar u_x-u_x\bar u)
\begin{pmatrix} 1 & 0 & 0 \\
0 & -1 & 0 \\
0 & 0 & 0 \end{pmatrix}.
\end{align}
Note that $\mathsf{U}$ and $\mathsf{V}$ are rapidly decaying as $|x| \to \infty$ if $u$ is, and that $L,Z$ obey the symmetries
\begin{subequations}\label{symmetry}
\begin{align}
& L(x,t,k) = -L^{\dagger}(x,t,\bar{k}),&& Z(x,t,k) = -Z^{\dagger}(x,t,\bar{k}),
	\\
& L(x,t,k)=\mathcal{A}\overline{L(x,t,-\bar{k})}\mathcal{A},&& Z(x,t,k)=\mathcal{A}\overline{Z(x,t,-\bar{k})}\mathcal{A},
\end{align}
\end{subequations}
where $A^{\dagger}$ denotes the complex conjugate transpose of a matrix $A$  and
$$\mathcal{A}=\begin{pmatrix} 0&1&0\\1&0&0\\0&0&1\end{pmatrix}.$$

\section{Proof of Theorem \ref{th1}}\label{th1proofsec}
Suppose $\rho_1 \in \mathcal{S}(\R)$. The associated jump matrix $v(x,t,k)$ defined in (\ref{vdef}) obeys the symmetries
\begin{align}\label{vsymm}
& v(x,t,k) = v^{\dagger}(x,t,\bar{k}) =\mathcal{A}\overline{v(x,t,-\bar{k})}\mathcal{A}, \qquad k \in \R.
\end{align}
In particular, $v$ is Hermitian and positive definite for each $k \in \R$. Hence the result of Zhou \cite{z-1989} implies that there exists a vanishing lemma for the RH problem for $m(x,t,k)$, i.e., the associated homogeneous RH problem has only the zero solution.

Defining the nilpotent matrices $w^{\pm}(x,t,k)$ by
\begin{align*}
w^{-} = \begin{pmatrix}\0_{2\times2}&\0_{2\times 1}\\\rho(k)e^{2ikx-8ik^3t}&0\end{pmatrix}, \quad
w^{+} = \begin{pmatrix}\0_{2\times2}&\rho^{\dag}(\bar{k})e^{-2ikx+8ik^3t}\\ \0_{1\times 2} & 0\end{pmatrix},
\end{align*}
we can write $v(x,t,k)= (v^-)^{-1} v^+$, where $v^{\pm} \doteq I \pm w^{\pm}$.
For $h \in L^2(\R)$, we define the Cauchy transform $\mathcal{C} h$ by
\begin{align}\label{Cauchytransform}
(\mathcal{C}h)(z) = \frac{1}{2\pi i} \int_\R \frac{h(s)}{s - z} ds, \qquad z \in \C \setminus \R,
\end{align}
and denote the nontangential boundary values of $\mathcal{C}f$ from the left and right sides of $\R$ by $\mathcal{C}_+ f$ and $\mathcal{C}_-f$, respectively. Then $\mathcal{C}_+$ and $\mathcal{C}_-$ are bounded operators on $L^2(\R)$ and $\mathcal{C}_+ - \mathcal{C}_- = I$.
Given two functions $w^\pm \in L^2(\R) \cap L^\infty(\R)$, we define the operator $\mathcal{C}_{w}: L^2(\R) + L^\infty(\R) \to L^2(\R)$ by
\begin{align}\label{Cwdef}
\mathcal{C}_{w}(f) = \mathcal{C}_+(f w^-) + \mathcal{C}_-(f w^+).
\end{align}
For each $(x,t) \in \R \times [0,\infty)$, we have $v^\pm \in C(\R)$ and $v^{\pm}, (v^{\pm})^{-1} \in I + L^2(\R) \cap L^\infty(\R)$. 
In view of the vanishing lemma, this implies (see e.g. \cite[Theorem 5.10]{LCarleson}) that $I - \mathcal{C}_w$ is an invertible bounded linear operator on $L^2(\R)$, and that the $3 \times 3$ matrix $L^2$-RH problem for $m$ has a unique solution $m(x,t,k)$ for each  $(x,t) \in \R^2$ given by
$$m = I + \mathcal{C}(\mu (w^+ + w^-)),$$
where
$$\mu = I + (I - \mathcal{C}_w)^{-1}\mathcal{C}_w I \in I + L^2(\R).$$
The smoothness and decay of $w^\pm$ together with the smooth dependence on $(x,t)$ implies that $m$ is a classical solution of the RH problem and that $m$ admits an expansion
\begin{align}\label{mexpansion}
m(x,t,k)=I+\frac{m_1(x,t)}{k}+\frac{m_2(x,t)}{k^2}+O(k^{-3}),\qquad k\to \infty,
\end{align}
where the coefficients $m_j(x,t)$ are smooth functions of $(x,t) \in \R^2$ (see e.g. \cite[Section 4]{hl-2018-2} for details in a similar situation).
Since $\rho_1 \in \mathcal{S}(\R)$, an application of the Deift-Zhou steepest descent method \cite{dz-1993} implies that $m$ and the coefficients $m_j$ have rapid decay as $|x| \to \infty$ for each $t$. In particular, the limit in (\ref{ulim}) exists for each $(x,t) \in \R^2$ and $u(x,t) = 2i(m_1(x,t))_{13}$ is a smooth function of $(x,t) \in \R^2$ with rapid decay as $|x| \to \infty$. 

\begin{lemma}
Define $u(x,t)$ by (\ref{ulim}). Then
\begin{align}\label{lax-m}
\begin{cases}
m_x+ik[\Lambda,m]=\mathsf{U}m,\\
m_t-4ik^3[\Lambda,m]=\mathsf{V}m.
\end{cases}
 \qquad (x,t) \in \R^2, \ k \in \C \setminus \R,
\end{align}
where $\mathsf{U}$  and $\mathsf{V}$ are defined in terms of $u(x,t)$ by \eqref{L1def} and \eqref{Z1def}, respectively.
\end{lemma}
\begin{proof}
The symmetries (\ref{vsymm}) of $v$ together with the uniqueness of the solution of the RH problem imply the following symmetries for $m$:
\begin{align}\label{msymm}
& m(x,t,k) = m^{\dagger}(x,t,\bar{k})^{-1} =\mathcal{A}\overline{m(x,t,-\bar{k})}\mathcal{A}.
\end{align}
In particular, the coefficient $m_1$ in (\ref{mexpansion}) satisfies 
$$m_1(x,t) = -m_1^{\dagger}(x,t) =-\mathcal{A}\overline{m_1(x,t)}\mathcal{A}.$$
It follows that the definition (\ref{ulim}) of $u(x,t)$ can be expressed as
\begin{align}\label{L11}
\mathsf{U}(x,t) =i[\Lambda,m_1(x,t)].
\end{align}

Define the operator $\mathbb{L}$ by
\begin{align}\label{mathbbLdef}
\mathbb{L} m \doteq m_x+i k [\Lambda,m]-\mathsf{U} m.
\end{align}
Substituting the expansion (\ref{mexpansion}) into \eqref{mathbbLdef}, we find
\begin{align*}
\mathbb{L}m= i[\Lambda, m_1]-\mathsf{U} +O(k^{-1}), \qquad k \to \infty.
\end{align*}
In view of (\ref{L11}), this implies that $\mathbb{L}m$ satisfies the following homogeneous RH problem:
\begin{itemize}
\item $\mathbb{L}m$ is analytic in $\C\setminus \R$ with continuous boundary values on $\R$;
  \item $(\mathbb{L}m)_+=(\mathbb{L}m)_- v$ for $k \in \R$;
  \item $\mathbb{L}m=O(k^{-1})$ as $k\to\infty$.
\end{itemize}
Thus, by the vanishing lemma, $\mathbb{L}m=0$. This proves the first equation in \eqref{lax-m}.

In order to prove the second equation in \eqref{lax-m}, we define the operator $\mathbb{Z}$ by
\begin{align}\label{mathbbZdef}
\mathbb{Z}m \doteq m_t-4ik^3[\Lambda,m]-k^2 A(x,t) m-k B(x,t) m-C(x,t)m,
\end{align}
where the matrices $A(x,t), B(x,t)$ and $C(x,t)$ are yet to be determined. Substituting the asymptotic expansion \eqref{mexpansion} into \eqref{mathbbZdef}, we find
\begin{align*}
\mathbb{Z}m=&\; \big(-4i[\Lambda,m_1]-A\big)k^2 + \big(-4i[\Lambda,m_2]-Am_1-B\big)k \\
& +\big(-4i[\Lambda,m_3]-Am_2-Bm_1-C\big) +O(k^{-1}), \qquad k\to\infty.
\end{align*}
Thus, we define $A, B, C$ by the equations
\begin{subequations}
\begin{align}
\label{Adef}A&=-4i[\Lambda,m_1],\\
\label{Bdef}B&=-4i[\Lambda,m_2]-Am_1,\\
\label{Cdef}C&=-4i[\Lambda,m_3]-A m_2 -B m_1.
\end{align}
\end{subequations}
If we can show that $A = \mathsf{V}^{(2)}$, $B = \mathsf{V}^{(1)}$, and $C = \mathsf{V}^{(0)}$, it will follow from the vanishing lemma that $\mathbb{Z}m=0$, which will prove the second equation in \eqref{lax-m}.

Comparing \eqref{L11} and \eqref{Adef}, we see that $A=-4\mathsf{U} = \mathsf{V}^{(2)}$, and then \eqref{Bdef} becomes
\begin{align}\label{Bnew}
B =4\mathsf{U}m_1-4i[\Lambda,m_2].
\end{align}
The terms of order $O(k^{-1})$ in the asymptotic expansion of the equation $\mathbb{L}m=0$ yield
\begin{align}\label{m1}
m_{1,x}+i[\Lambda,m_2]=\mathsf{U} m_1.
\end{align}
Comparing \eqref{m1} with \eqref{Bnew}, it follows that $B = 4m_{1,x}$.
Given a $3\times 3$ matrix
\begin{align*}
A=\begin{pmatrix}a_{11} &a_{12}&a_{13}\\a_{21} &a_{22}&a_{23}\\a_{31} &a_{32}&a_{33}\\
\end{pmatrix},
\end{align*}
let us write $A=A^{(o)}+A^{(d)}$, where
\begin{align*}
A^{(o)}=\begin{pmatrix}0 &0&a_{13}\\0 &0&a_{23}\\a_{31} &a_{32}&0\\
\end{pmatrix},\quad A^{(d)}= \begin{pmatrix}a_{11} &a_{12}&0\\a_{21} &a_{22}&0\\0 &0&a_{33}\\
\end{pmatrix}.
\end{align*}
Equation \eqref{L11} can then be written as $m_1^{(o)}=-\frac{i}{2}\Lambda \mathsf{U}$, and hence
\begin{align}\label{m1off}
m_{1,x}^{(o)}=-\frac{i}{2}\Lambda \mathsf{U}_x.
\end{align}
According to \eqref{m1}, we have
\begin{align}\label{m1diag}
m_{1,x}^{(d)}=\mathsf{U} m_1^{(o)}=-\frac{i}{2}\mathsf{U} \Lambda \mathsf{U}.
\end{align}
Equations \eqref{m1off} and \eqref{m1diag} imply 
$$B= 4m_{1,x}=-2i(\Lambda \mathsf{U}_x+\mathsf{U} \Lambda \mathsf{U})= \mathsf{V}^{(1)}.$$

It only remains to prove that $C = \mathsf{V}^{(0)}$. The terms of order $O(k^{-2})$ in the expansion of the equation $\mathbb{L}m=0$ yield
\begin{align}\label{m2}
m_{2,x}+i[\Lambda,m_3]=\mathsf{U} m_2.
\end{align}
It follows that $C =4m_{2,x}-Bm_1$. On the other hand, \eqref{m1} and (\ref{m2}) imply
\begin{align*}
m_{2}^{(o)}=- \frac{i}{2}\Lambda\big(\mathsf{U} m_1^{(d)}-m_{1,x}^{(o)}\big), \qquad m_{2,x}^{(d)}=\mathsf{U} m_2^{(o)}.
\end{align*}
We conclude that
\begin{align*}
C =4m_{2,x}-Bm_1&=-B m_{1}^{(o)}-\frac{i}{2}\Lambda \mathsf{U} B+2 i \Lambda m_{1,xx}^{(o)}
= \mathsf{V}^{(0)},
\end{align*}
which proves the lemma.
\end{proof}

The compatibility condition of \eqref{lax-m} shows that $u(x,t)$ satisfies \eqref{sseq}. 
The proof of Theorem \ref{th1} is complete.

\section{Proof of Theorem \ref{th2}}\label{th2proofsec}
Let $\rho_1 \in \mathcal{S}(\R)$ and let $u(x,t)$ be the associated solution of (\ref{sseq}) defined by (\ref{ulim}).
Our goal is to find the asymptotics of $u(x,t)$ in the sector $\mathcal{P}$ defined by
\begin{align}\label{Pdef}
\mathcal{P}= \{(x,t) \in \R^2 \, | \, |x| \leq M t^{1/3}, t\geq 1\},
\end{align}
where $M > 0$ is a constant. Let 
$$\mathcal{P}_\geq \doteq \mathcal{P} \cap \{x \geq 0\} \quad  \text{and} \quad \mathcal{P}_\leq \doteq \mathcal{P} \cap \{x \leq 0\}$$
denote the right and left halves of $\mathcal{P}$. For conciseness, we will give the proof of the asymptotic formula (\ref{uasymptotics}) for $(x,t) \in \mathcal{P}_\geq$; the case when $(x,t) \in \mathcal{P}_\leq$ can be handled in a similar way but requires some (minor) changes in the arguments (see \cite{CL2019} for the required changes in the case of the mKdV equation).

The jump matrix $v(x,t,k)$ defined in \eqref{vdef} involves the exponentials $e^{\pm t\Phi(\zeta,k)}$, where $\Phi(\zeta,k)$ is defined by
\begin{align}
\Phi(\zeta,k) \doteq 2ik\zeta-8ik^3 \quad \text{with} \quad \zeta \doteq x/t.
\end{align}
Suppose $(x,t) \in \mathcal{P}_\geq$. Then there are two real critical points (i.e., solutions of $\partial \Phi/\partial k = 0$) located at the points $\pm k_0$, where (see Figure \ref{criticalpts.pdf})
$$k_0 = \sqrt{\frac{x}{12 t}} \geq 0.$$ 
As $t \to \infty$, the critical points $\pm k_0$ approach $0$ at least as fast as $t^{-1/3}$, i.e., $0 \leq k_0 \leq C t^{-1/3}$.

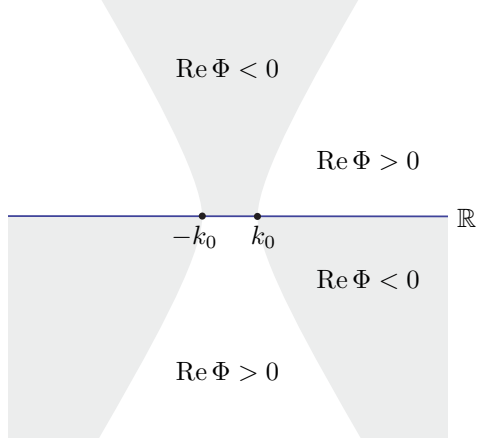
\begin{figure}
\begin{center}
    \begin{overpic}[width=.4\textwidth]{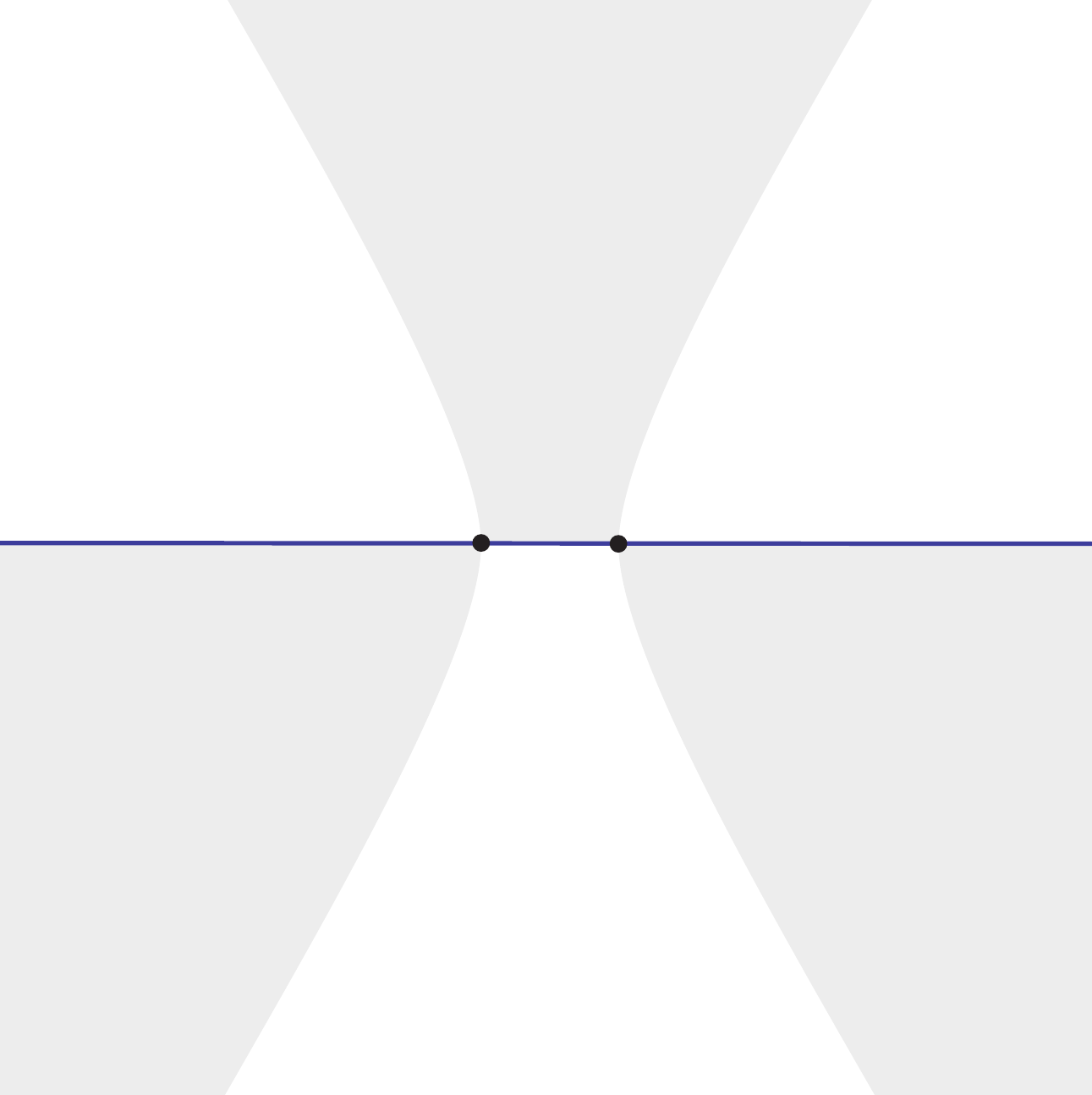}
      \put(102,48){\small $\R$}
      \put(55,44){\small $k_0$}
      \put(37,44){\small $-k_0$}
      \put(38,82){\small $\re \Phi < 0$}
      \put(70,61){\small $\re \Phi > 0$}
      \put(70,34){\small $\re \Phi < 0$}
      \put(38,13){\small $\re \Phi > 0$}  
     \end{overpic}
\caption{The critical points $\pm k_0$ in the complex $k$-plane together with the regions where $\re \Phi > 0$ (white) and $\re \Phi < 0$ (shaded).}
\label{criticalpts.pdf}
     \end{center}
\end{figure}

\subsection{Analytic approximation}
We first decompose $\rho = (\rho_1, \rho_2)$ into an analytic part $\rho_a$ and a small remainder $\rho_r$.
Let $N \geq 1$ be an integer. 
Let $\Gamma^{(1)} \subset \C$ denote the contour 
$$\Gamma^{(1)} = \R \cup \Gamma_1^{(1)} \cup \Gamma_2^{(1)},$$
where 
\begin{align*}
& \Gamma_1^{(1)} = \{k_0 + re^{\frac{\pi i}{6}}\, | \, r \geq 0\} \cup \{-k_0 + re^{\frac{5\pi i}{6}}\, | \, r \geq 0\},
	\\
& \Gamma_2^{(1)} = \{k_0 + re^{-\frac{\pi i}{6}}\, | \, r \geq 0\} \cup \{-k_0 + re^{-\frac{5\pi i}{6}}\, | \, r \geq 0\}.
\end{align*}	
We orient $\Gamma^{(1)}$ to the right and let $V$ (resp. $V^*$) denote the open subset between $\Gamma_1^{(1)}$ (resp. $\Gamma_2^{(1)}$) and the real line, see Figure \ref{Gamma1.pdf}.

\begin{lemma}[Analytic approximation]\label{decompositionlemmaleq}
There exists a decomposition
\begin{align*}
& \rho_1(k) = \rho_{1,a}(x, t, k) + \rho_{1,r}(x, t, k), \qquad k \in (-\infty, -k_0)\cup (k_0, \infty),
\end{align*}
where the functions $\rho_{1,a}$ and $\rho_{1,r}$ have the following properties:
\begin{enumerate}[$(a)$]
\item For each $(x,t) \in \mathcal{P}_\geq$, $\rho_{1,a}(x, t, k)$ is defined and continuous for $k \in \bar{V}$ and analytic for $k \in V$.

\item The function $\rho_{1,a}$ obeys the following estimates uniformly for $(x,t) \in \mathcal{P}_\geq$:
\begin{align*}
& |\rho_{1,a}(x, t, k)| \leq \frac{C}{1 + |k|} e^{\frac{t}{4}|\re \Phi(\zeta,k)|}, \qquad
  k \in \bar{V}, 
\end{align*}
and
\begin{align*}
\bigg|\rho_{1,a}(x, t, k) - \sum_{j=0}^N \frac{\rho_1^{(j)}(k_0)}{j!} (k-k_0)^j\bigg| \leq C |k-k_0|^{N+1} e^{\frac{t}{4}|\re \Phi(\zeta,k)|}, \qquad k \in \bar{V}.
\end{align*}

\item The $L^1$ and $L^\infty$ norms of $\rho_{1,r}(x, t, \cdot)$ on $(-\infty, -k_0)\cup (k_0, \infty)$ are $O(t^{-N})$ as $t \to \infty$ uniformly for $(x,t) \in \mathcal{P}_\geq$.

\end{enumerate}
\end{lemma}
\begin{proof}
See \cite[Lemma 5.1]{CL2019}.
\end{proof}

\begin{figure}
\begin{center}
\begin{overpic}[width=.5\textwidth]{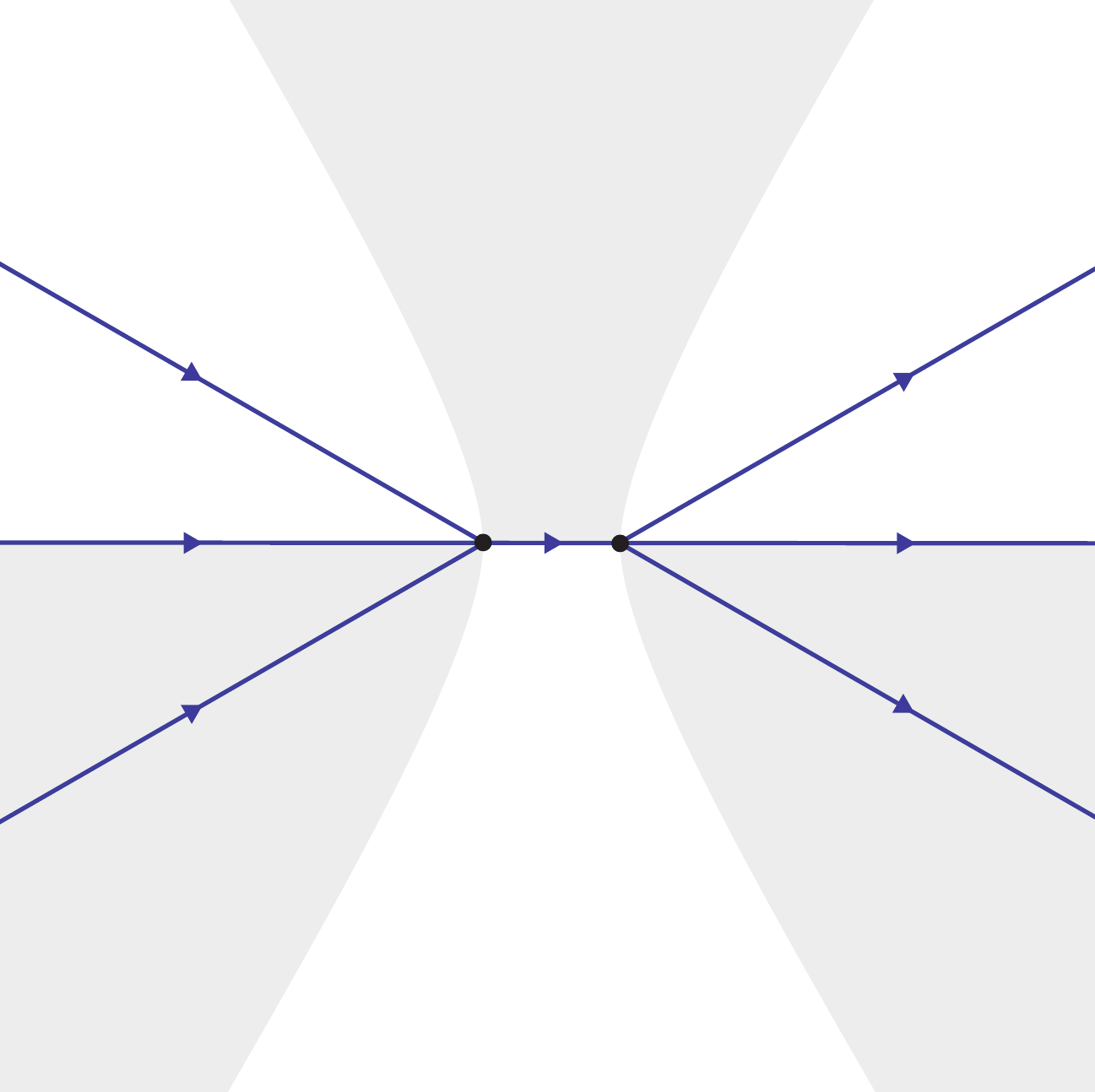}
      \put(103,48.5){\small $\Gamma^{(1)}$}
      \put(80,67){\small $1$}
      \put(18,67){\small $1$}
      \put(81,30){\small $2$}
      \put(18,30){\small $2$}
      \put(17,45.5){\small $3$}
      \put(81.5,45.5){\small $3$}
      \put(49,52){\small $4$}
      \put(55,45.5){\small $k_0$}
      \put(40,45.5){\small $-k_0$}
      \put(8,57){\small $V$}
      \put(8,39){\small $V^*$}
      \put(89,57){\small $V$}
      \put(89,39){\small $V^*$}
      \put(41,84){\small $\re \Phi < 0$}
    \end{overpic}
     \caption{The sets $V$ and $V^*$ and the contour  $\Gamma^{(1)}$.}
\label{Gamma1.pdf}
     \end{center}
\end{figure}
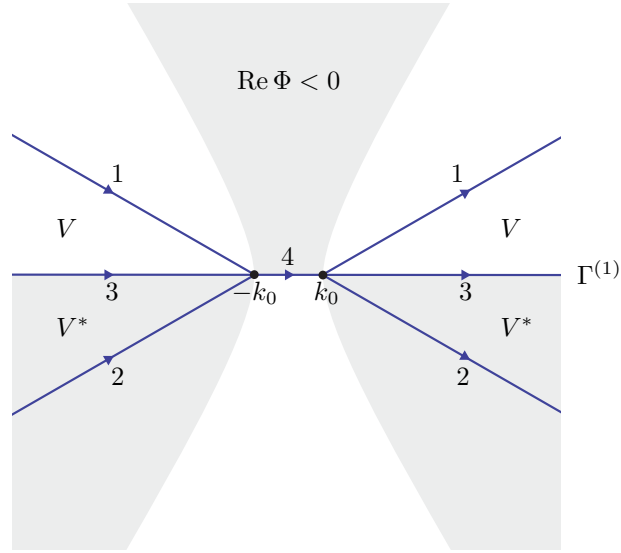

Letting $\rho_{2,a}(k) \doteq \overline{\rho_{1,a}(-\bar{k})}$ and $\rho_{2,r}(k) \doteq \overline{\rho_{1,r}(-\bar{k})}$, we obtain a decomposition $\rho = \rho_a + \rho_r$ of $\rho$ by setting 
$$\rho_a(k) \doteq \begin{pmatrix}\rho_{1,a}(k) & \rho_{2,a}(k)\end{pmatrix}, \quad 
\rho_r(k) \doteq \begin{pmatrix}\rho_{1,r}(k) & \rho_{2,r}(k)\end{pmatrix}.$$

\subsection{Opening of the lenses}
The jump matrix $v$ enjoys the factorization
\begin{align}v(x,t,k)=
\begin{pmatrix}\I_{2\times2}&\0_{2\times1}\\\rho e^{t\Phi}&1\end{pmatrix} \begin{pmatrix}\I_{2\times2}& {\bf \rho} ^{\dagger} e^{-t\Phi}\\\0_{1\times2}&1\end{pmatrix}.
\end{align}
It follows that $m$ satisfies the RH problem in Theorem \ref{th1} if and only if the function $m^{(1)}$ defined by
\begin{align}
m^{(1)}(x,t,k)=\begin{cases} m(x,t,k) \begin{pmatrix}\I_{2\times2} & -\rho^{\dag}_a(x,t,\bar{k}) e^{-t\Phi}\\\0_{1\times2}&1\end{pmatrix},  & k \in V,
	\\
m(x,t,k) \begin{pmatrix}\I_{2\times2}&\0_{2\times1} \\ \rho_a(x,t,k) e^{t\Phi}&1\end{pmatrix},  & k \in V^*,
	\\
m(x,t,k),  & \text{elsewhere}, 
\end{cases}
\end{align}
satisfies the RH problem
\begin{itemize}
\item $m^{(1)}(x,t,\cdot)$ is analytic in $\C\setminus \R$ with continuous boundary values on $\Gamma \setminus \{\pm k_0\}$;
  \item $m^{(1)}_+=m^{(1)}_- v^{(1)}$ for $k \in \Gamma \setminus \{\pm k_0\}$;
  \item $m^{(1)} =I + O(k^{-1})$ as $k\to\infty$;
  \item $m^{(1)} = O(1)$ as $k \to \pm k_0$;
\end{itemize}
where the jump matrix $v^{(1)}(x,t,k)$ is given by
\begin{align}\label{vshare}
v^{(1)} =\begin{cases}
v_1^{(1)} = \begin{pmatrix}\I_{2\times2} & \rho^{\dag}_{a}(x,t,\bar{k}) e^{-t\Phi}\\\0_{1\times2}&1\end{pmatrix},
\quad &k\in \Gamma_1^{(1)},
	\\
v_2^{(1)} = \begin{pmatrix}\I_{2\times2}&\0_{2\times1} \\ \rho_{a}(x,t,k) e^{t\Phi}&1\end{pmatrix},\quad &k\in \Gamma_2^{(1)},
	\\
v_3^{(1)} = \begin{pmatrix}\I_{2\times2}&\0_{2\times1}\\ \rho e^{t\Phi}&1\end{pmatrix} \begin{pmatrix}\I_{2\times2}&  \rho^{\dagger} e^{-t\Phi}\\\0_{1\times2}&1\end{pmatrix}, \quad & k\in (-k_0,k_0),\\
v_4^{(1)}=\begin{pmatrix}\I_{2\times2}&\0_{2\times1}\\ \rho_r e^{t\Phi}&1\end{pmatrix} \begin{pmatrix}\I_{2\times2}& \rho_r^{\dagger} e^{-t\Phi}\\\0_{1\times2}&1\end{pmatrix}, \quad & k\in \R \setminus [-k_0,k_0].\\
\end{cases}
\end{align}
Note that $v^{(1)}$ and $m^{(1)}$ obey the same symmetries (\ref{vsymm}) and (\ref{msymm}) as $v$ and $m$.

\subsection{Local model}
Let us introduce new variables $y$ and $z$ by
\begin{align}
y \doteq \frac{x}{(3t)^{1/3}}, \quad z \doteq (3t)^{1/3}k,
\end{align}
so that
\begin{align}
t\Phi(\zeta,k)=2i\bigg(y z-\frac{4 z^3}{3}\bigg).
\end{align}
Fix $\epsilon > 0$ and let $D_\epsilon(0) = \{k \in \C \, | \, |k| < \epsilon\}$.
Let $\mathcal{Z}^\epsilon = (\Gamma^{(1)} \cap D_\epsilon(0))\setminus ((-\infty,-k_0) \cup (k_0, \infty))$, see Figure \ref{calZ.pdf}. Let $Z$ denote the contour defined in (\ref{ZdefIVg}) with $z_0 \doteq (3t)^{1/3}k_0 = \sqrt{y}/2 \geq 0$.
The map $k \mapsto z$ maps $\mathcal{Z}^\epsilon$ onto $Z \cap \{|z| < (3t)^{1/3}\epsilon\}$ and we have $\mathcal{Z}^\epsilon = \cup_{j=1}^5\mathcal{Z}_j^\epsilon$, where $\mathcal{Z}_j^\epsilon$ denotes the inverse image of $Z_j\cap \{|z| < (3t)^{1/3}\epsilon\}$ under this map.

\begin{figure}
\begin{center}
 \begin{overpic}[width=.4\textwidth]{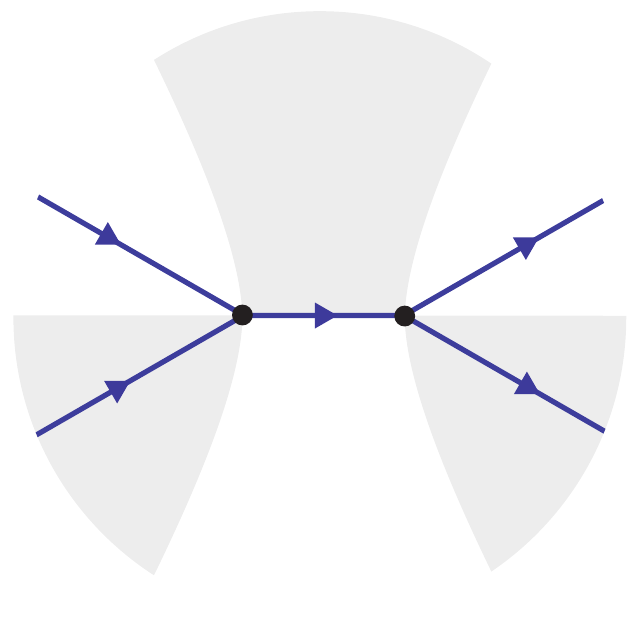}
 \put(77,65.5){\small $\mathcal{Z}_1^\epsilon$}
 \put(18,65){\small $\mathcal{Z}_2^\epsilon$}
 \put(18,32){\small $\mathcal{Z}_3^\epsilon$}
 \put(77,32){\small $\mathcal{Z}_4^\epsilon$}
 \put(47.5,54){\small $\mathcal{Z}_5^\epsilon$}
 \put(61,44){\small $k_0$}
 \put(33,44){\small $-k_0$}
 \end{overpic}
\caption{The contour $\mathcal{Z}^\epsilon = \cup_{j=1}^5\mathcal{Z}_j^\epsilon$.}
\label{calZ.pdf}   
\end{center}
\end{figure}
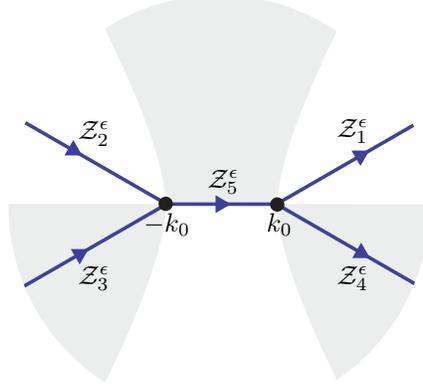

Let $p$ denote the $N$th order Taylor polynomial of $\rho$ at $k = 0$, i.e., 
\begin{align}\label{pNdef}
  p(t,z) \doteq \sum_{j=0}^N \frac{\rho^{(j)}(0)}{j!} k^j 
  = \sum_{j=0}^N \frac{\rho^{(j)}(0)}{j!3^{j/3}} \frac{z^j}{t^{j/3}}.
\end{align}
For large $t$ and fixed $z$, the jump matrices $\{v_{j}^{(1)}\}_1^4$ can be approximated as follows:
\begin{align}\nonumber
v_1^{(1)} &\approx \begin{pmatrix}\I_{2\times 2}& p^{\dag}(t, \bar{z})e^{-2i(yz-\frac{4 z^3}{3})} \\ \0_{1\times 2}&1\end{pmatrix},
	\\\nonumber
v_2^{(1)} & \approx \begin{pmatrix} \I_{2\times 2} & \0_{2\times 1}\\ p(t,z) e^{2i(yz-\frac{4 z^3}{3})} & 1\end{pmatrix},
	\\\nonumber
v_3^{(1)}&\approx \begin{pmatrix}\I_{2\times 2}& p^{\dag}(t, \bar{z})e^{-2i(yz-\frac{4 z^3}{3})} \\ p(t,z)e^{2i(yz-\frac{4 z^3}{3})} &1+ p(t,z) p^{\dag}(t,\bar{z}) \end{pmatrix},
	\\ \label{vapproximations}
v_4^{(1)} &\approx I.
\end{align}
Thus we expect that $m^{(1)}$ approaches the solution $m_0(x,t,k)$ defined by
\begin{align}\label{m0defIVg}
m_0(x,t,k) \doteq m^Z(y, t, z_0, z)
\end{align}
for large $t$, where $m^Z(y, t, z_0, z)$ is the solution of the model RH problem of Lemma \ref{ZlemmaIVg} with $z_0 = \sqrt{y}/2$ and $p(t,z)$ given by (\ref{pNdef}).
If $(x,t) \in \mathcal{P}_\geq$, then $(y,t,z_0) \in \mathbb{P}$, where $\mathbb{P}$ is the parameter subset defined in (\ref{parametersetIVg}). Thus Lemma \ref{ZlemmaIVg} ensures that $m_0$ is well-defined by (\ref{m0defIVg}). By (\ref{mZsymm}), $m_0$ obeys the same symmetries (\ref{msymm}) as $m$.

\begin{figure}
\begin{center}
\begin{overpic}[width=.4\textwidth]{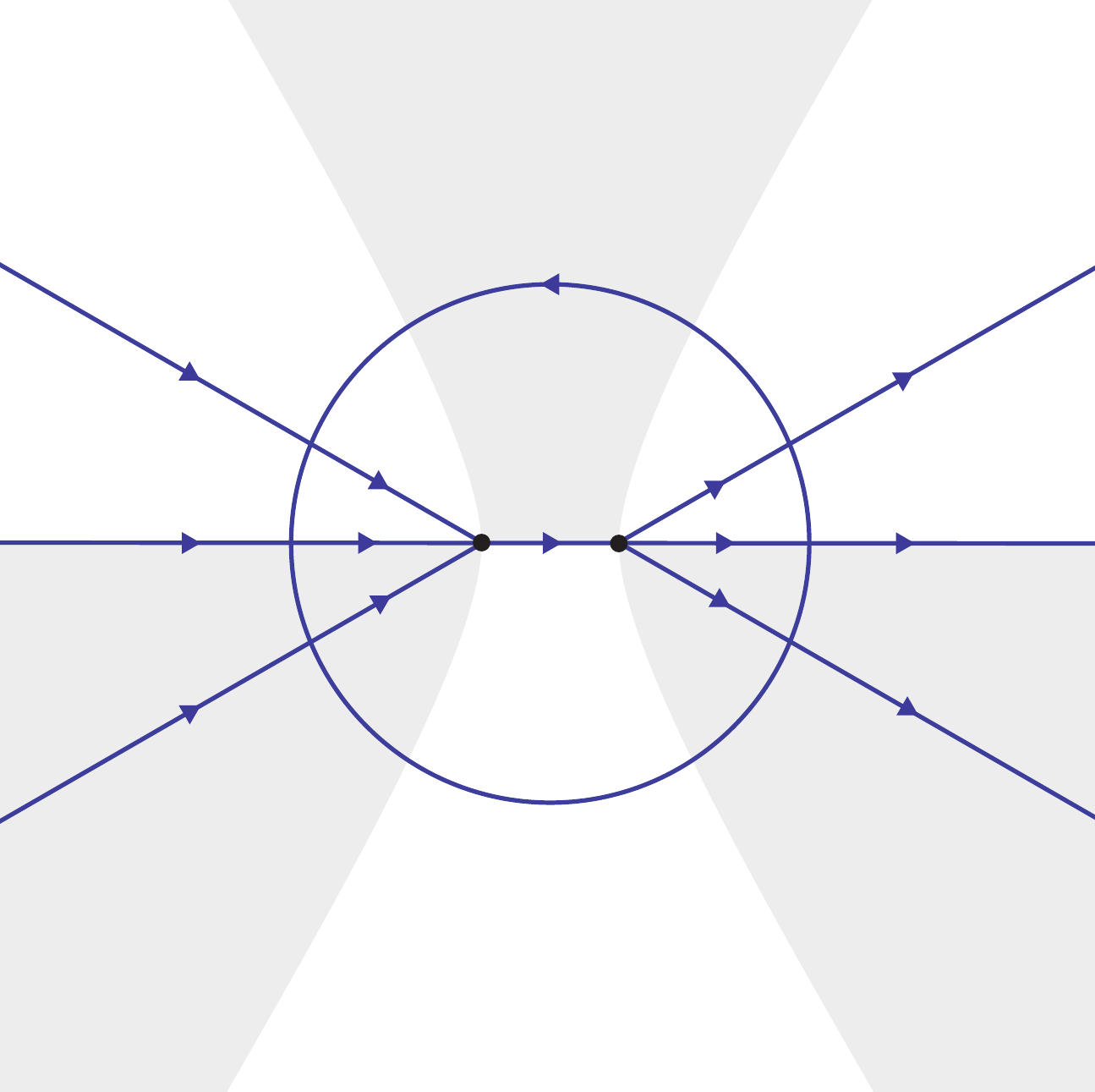}
      \put(103,48){\small $\hat{\Gamma}$}
      \put(55,44){\small $k_0$}
      \put(39,44){\small $-k_0$}
    \end{overpic}
\caption{The contour  $\hat{\Gamma}$.}
\label{Gammahat.pdf}
     \end{center}
\end{figure}

\subsection{The solution $\hat{m}$}
Fix $\epsilon >0$. Let $\hat{\Gamma} \doteq \Gamma^{(1)} \cup \partial D_\epsilon(0)$ and assume that the boundary of $D_\epsilon(0)$ is oriented counterclockwise, see Figure \ref{Gammahat.pdf}. 
Define $\hat{m}(x,t,k)$ by
$$\hat{m} = \begin{cases} m^{(1)} m_0^{-1}, & k \in D_\epsilon(0), \\
m^{(1)}, & k \in \C \setminus D_\epsilon(0),
\end{cases}.$$
Then $\hat{m}$ satisfies a small-norm RH problem with jump $\hat{m}_+=\hat{m}_- \hat{v}$ across $\hat{\Gamma}$, where the jump matrix $\hat{v}$ is given by
\begin{align}\label{vhatdef}
\hat{v} 
=  \begin{cases}
 m_{0-} v^{(1)} m_{0+}^{-1}, & k \in \hat{\Gamma} \cap D_\epsilon(0), \\
m_0^{-1}, & k \in \partial D_\epsilon(0), \\
v^{(1)},  & k \in \hat{\Gamma} \setminus \overline{D_\epsilon(0)}.
\end{cases}
\end{align}
Using Lemma \ref{ZlemmaIVg}, the rest of the proof proceeds as in the case of the mKdV equation (see \cite{CL2019}) and we only give a brief outline. 
Let $\hat{\mathcal{C}}$ be the Cauchy operator associated with $\hat{\Gamma}$ and let $\hat{\mathcal{C}}_{\hat{w}}f \doteq \hat{\mathcal{C}}_-(f\hat{w})$. Then
\begin{align}\label{mhatrepresentation}
\hat{m}(x, t, k) = I + \frac{1}{2\pi i}\int_{\hat{\Gamma}} (\hat{\mu} \hat{w})(x, t, s) \frac{ds}{s - k},
\end{align}
where $\hat{w} = \hat{v} - I$ and $\hat{\mu}(x, t, k) \in I + L^2(\hat{\Gamma})$ is defined by $\hat{\mu} = I + (I - \hat{\mathcal{C}}_{\hat{w}})^{-1}\hat{\mathcal{C}}_{\hat{w}}I$.
The expansion (\ref{mYasymptoticsIVg}) of $m^Z$ translates into expansions of $\hat{w}$ and $\hat{\mu}$ in powers of $t^{-1/3}$ with coefficients which are functions of $y$. 
It follows that there are smooth functions $h_j(y)$ such that
\begin{align*}
 \lim_{k\to \infty} k(m(x,t,k) - I)
& = \lim_{k\to \infty} k(\hat{m}(x,t,k) - I)
= - \frac{1}{2\pi i}\int_{\hat{\Gamma}} \hat{\mu}(x,t,k) \hat{w}(x,t,k) dk
	\\
& =  - \sum_{j=1}^{N} \frac{h_j(y)}{t^{j/3}} + O\big(t^{-\frac{N+1}{3}}\big), \qquad t \to \infty,
\end{align*}
uniformly for $(x,t) \in \mathcal{P}_\geq$, where $h_1(y)$ is the coefficient of $t^{-1/3}$ in the large $t$ expansion of
\begin{align*}
\frac{1}{2\pi i}\int_{\partial D_\epsilon(0)} \hat{w} dk
= -\frac{1}{2\pi i}\int_{\partial D_\epsilon(0)} \frac{m_{10}^Z(y)}{(3t)^{1/3}k} dk + O(t^{-2/3})
= -\frac{m_{10}^Z(y)}{(3t)^{1/3}} + O(t^{-2/3}).
\end{align*}
Hence, $u(x,t)=2i\lim_{k\to\infty}(k m(x,t,k))_{13} $ has an expansion of the form (\ref{uasymptotics}) with leading coefficient given by
\begin{align*}
u_1(y) = 2i \frac{m_{10}^Z(y)}{3^{1/3}} = i\frac{u_P(y;s)}{3^{1/3}\sqrt{2}}.
\end{align*}
This completes the proof of Theorem \ref{th2}.

\section{Proof of Theorem \ref{th3}}\label{th3proofsec}
Let $u_0 \in \mathcal{S}(\R)$ and suppose $u(x,t)$ is a smooth solution of (\ref{sseq}) with initial data $u(x,0) = u_0(x)$ and with rapid decay as $|x| \to \infty$.
If $\psi$ satisfies the Lax pair equations (\ref{psilax}), then the eigenfunction $\Psi$ defined by $\psi=\Psi e^{-i(kx-4k^3t)\Lambda}$ satisfies 
\begin{align}\label{Psilax}
\begin{cases}
\Psi_x+ik[\Lambda,\Psi] =\mathsf{U}\Psi,
	\\
\Psi_t-4ik^3[\Lambda,\Psi] =\mathsf{V}\Psi.
\end{cases}
\end{align}
We define two solutions $\{\Psi_j\}_1^2$ of (\ref{Psilax}) as the unique solutions of the integral equations
\begin{subequations}\label{Psijdef}
\begin{align}
 & \Psi_1(x,t,k) = I + \int_{-\infty}^x e^{ik(x'-x)\hat \Lambda} (\mathsf{U}\Psi_1)(x',t',k)dx', 
  	\\
 & \Psi_2(x,t,k) = I - \int_x^{\infty} e^{ik(x'-x)\hat \Lambda} (\mathsf{U}\Psi_2)(x',t',k)dx'.	
\end{align}
\end{subequations}

Let $\C_\pm \doteq \{\im k \gtrless 0\}$. 
The third columns of the matrix equations (\ref{Psijdef}) involves the exponential
$e^{2ik(x' - x)}$. Since the equations in (\ref{Psijdef}) are Volterra integral equations, it follows that the third column vectors of $\Psi_1$ and $\Psi_2$ are bounded and analytic for $k \in \C_-$ and $k \in \C_+$, respectively, with smooth extensions to $\R$. Similar considerations apply to the first and second columns; thus
\begin{align*}
& \Psi_1(x,t,k) \; \text{is bounded and analytic for} \; k \in (\C_+, \C_+, \C_-),
	\\
& \Psi_2(x,t,k)\; \text{is bounded and analytic for} \; k \in (\C_-, \C_-, \C_+),
\end{align*}
where $k \in (\C_+, \C_+, \C_-)$ indicates that the first, second, and third columns of the equation are valid for $k$ in $\C_+, \C_+$ and $\C_-$, respectively. 
Moreover, for each $t$ and each $j \geq 0$, there are bounded functions $f_-(x)$ and $f_+(x)$ of $x \in \R$ with rapid decay as $x \to -\infty$ and $x \to +\infty$, respectively, such that
\begin{subequations}\label{Xest}
\begin{align}\label{Xesta}
& \bigg|\frac{\partial^j}{\partial k^j}\big(\Psi_1(x,t,k) - I\big) \bigg| \leq
f_-(x), \qquad k \in (\bar{\C}_+, \bar{\C}_+, \bar{\C}_-), \ x \in \R,
	\\ \label{Xestb}
& \bigg|\frac{\partial^j}{\partial k^j}\big(\Psi_2(x,t,k) - I\big) \bigg| \leq
f_+(x), \qquad k \in (\bar{\C}_-, \bar{\C}_-, \bar{\C}_+), \ x \in \R. 
\end{align}
\end{subequations}
As $k \to \infty$, $\Psi_1$ and $\Psi_2$ have asymptotic expansions of the form
\begin{align}\label{Psijexpansions} 
  \Psi_j(x,t,k) \sim I + \sum_{n=1}^\infty \frac{\Psi_j^{(n)}(x,t)}{k^n}, \qquad j = 1,2,
\end{align}
where the coefficients $\Psi_j^{(n)}(x,t)$ are smooth bounded functions of $x$ for each $t$ and the expansion is valid uniformly for $k \in (\bar{\C}_+, \bar{\C}_+, \bar{\C}_-)$ if $j = 1$ and for $k \in (\bar{\C}_-, \bar{\C}_-, \bar{\C}_+)$ if $j = 2$. The above properties follow from an analysis of the Volterra equations (\ref{Psijdef}); see e.g. \cite{DT1979} or Theorem 3.1 in \cite{HLNonlinearFourier} for similar proofs.

The symmetries in \eqref{symmetry} imply that (cf. (\ref{msymm}))
\begin{align}\label{Psijsymm}
\Psi_j(x,t,k) =\Psi^{\dagger}_j(x,t,\bar{k})^{-1} =\mathcal{A}\overline{\Psi_j(x,t,-\bar{k})}\mathcal{A}, \qquad j=1,2.
\end{align}
Moreover, the tracelessness of $\mathsf{U}$ and $\mathsf{V}$ shows that $\det \Psi_j \equiv 1$ for $j = 1,2$. 
Indeed, the solution $\psi_j$ of (\ref{psilax}) given by $\psi_j = \Psi_j e^{-i(kx - 4k^3 t)\Lambda}$ satisfies
$$\begin{cases}
(\det \psi_j)_x = \tr(\psi_{jx} \psi_j^{-1}) \det \psi_j = -ik \det \psi_j, \\
(\det \psi_j)_t = \tr(\psi_{jt} \psi_j^{-1}) \det \psi_j = 4ik^3 \det \psi_j.
\end{cases}$$
Hence $\det \psi_j = c_j e^{-i(kx - 4k^3 t)}$ for some constant $c_j \in \C$. Thus, for each $j$, $\det \Psi_j(x,t,k)$ is independent of $(x,t)$; evaluation at $x = \pm \infty$ shows that $\det \Psi_j(x, t, k) \equiv 1$.

Define the $3 \times 3$-matrix valued spectral function $s(k)$ by
\begin{align}\label{sdef2}
  \Psi_2(x,t,k) = \Psi_1(x,t,k)e^{-i(kx-4k^3t)\hat{\Lambda}} s(k), \qquad x \in \R, \ k \in \R.
\end{align}
Letting $X(x,k) \doteq \Psi_2(x,0,k)$, we see that $s(k)$ can be expressed as in (\ref{sdef}). Since $\det \Psi_j \equiv 1$, (\ref{sdef2}) yields $\det s \equiv 1$. 
By (\ref{Psijsymm}), we have
\begin{align}\label{ssymm}
  s(k)=s^{\dagger}(\bar{k})^{-1} =\mathcal{A}\overline{s(-\bar{k})}\mathcal{A}.
\end{align}

Define $\rho_1(k)$ in terms of $s(k)$ by (\ref{rho1def}).
By assumption, $s_{33}(k)$ is nonzero for $\im k \geq 0$.

\begin{lemma}\label{rho1lemma}
The reflection coefficient $\rho_1(k)$ belongs to the Schwartz class $\mathcal{S}(\R)$. 
\end{lemma}
\begin{proof}
The expression in (\ref{sdef}) for the $(ij)$th entry of $s(k)$ involves the exponential factor $e^{ikx(\lambda_i - \lambda_j)}$, where $\lambda_1 = \lambda_2 = - \lambda_3 = 1$.
It follows from the properties of $\Psi_2$ and $\mathsf{U}$ that $s(k)$ is a smooth function of $k \in \R$ and that the $(33)$-entry $s_{33}$ admits an analytic continuation to the upper half-plane. It also follows (by replacing $X$ in (\ref{sdef}) by its large $k$ expansion and integrating by parts repeatedly in the resulting expression) that $s_{13}, s_{23}, s_{31}, s_{32}$ have rapid decay as $|k| \to \infty$.  
For the diagonal element $s_{33}(k)$, the exponential factor is absent from the integral in (\ref{sdef}), and substituting in the large $k$ expansion of $X$ we instead obtain
$$s_{33}(k) \sim 1 + \sum_{n=1}^\infty \frac{s_{33}^{(n)}}{k^n}, \qquad k \to \infty,$$
uniformly for $k \in \bar{\C}_+$ for some coefficients $\{s_{33}^{(n)}\} \subset \C$. The lemma follows. 
\end{proof}


Let $s_{ij}^*(k) = \overline{s_{ij}(\bar{k})}$ denote the Schwartz conjugate of $s_{ij}(k)$, $i,j = 1,2,3$.
Let $[A]_j$ denote the $j$th column of a matrix $A$.

\begin{lemma}
The function $m(x,t,k)$ defined by
\begin{align}\label{mPsidef}
m = \begin{cases}
\big([\Psi_1]_1, [\Psi_1]_2, \frac{[\Psi_2]_3}{s_{33}} \big), & \im k > 0, \\
\big(\frac{s_{22}[\Psi_2]_1 - s_{21}[\Psi_2]_2}{s_{33}^*}, 
\frac{-s_{12} [\Psi_2]_1 + s_{11} [\Psi_2]_2}{s_{33}^*},
 [\Psi_1]_3 \big), & \im k < 0.
 \end{cases}
\end{align}
satisfies the RH problem of Theorem \ref{th3} with $\rho_1(k)$ given by (\ref{rho1def}).
\end{lemma}
\begin{proof}
We saw in the proof of Lemma \ref{rho1lemma} that $s_{33}$ admits an analytic continuation to the upper half-plane. A similar argument shows that $s_{11}, s_{12}, s_{21}, s_{22}$ admit analytic continuations to the lower half-plane. Hence $m$ is well-defined by (\ref{mPsidef}) and the properties of $\Psi_1$, $\Psi_2$ together with the assumption that $s_{33}(k) \neq 0$ for $\im k \geq 0$ imply that $m(x,t,k)$ is analytic for $k \in \C\setminus \R$ with continuous boundary values on $\R$ from above and below.
The jump $m_+ =m_- v$ across $\R$ is a consequence of a long but straightforward computation which uses (\ref{sdef2}), the symmetries (\ref{ssymm}) of $s$, and the fact that $\det s = 1$. 
Finally, the normalization condition $m(x,t,k) = I+O(k^{-1})$ follows from the large $k$ behavior of $\Psi_1$, $\Psi_2$, and $s$.
\end{proof}

In view of Theorem \ref{th2}, the next lemma completes the proof of Theorem \ref{th3}.

\begin{lemma}
The solution $u(x,t)$ is given by (\ref{ulim}). 
\end{lemma}
\begin{proof}
Substituting the expansions (\ref{Psijexpansions}) into (\ref{Psilax}), we find that
\begin{align}\label{recoveru}
 u(x,t) = 2i\lim_{k\to \infty} \big(k \Psi_j(x,t,k)\big)_{13}, \qquad (x,t) \in \R^2, \ j = 1,2.
\end{align}
The lemma then follows from the definition (\ref{mPsidef}) of $m$ and the fact that $s_{33}(k) = 1 + O(k^{-1})$ as $k \to \infty$.
\end{proof}

\begin{remark}[Motivation for (\ref{mPsidef})]\upshape
The form of the expression (\ref{mPsidef}) for $m$ can be motivated as follows. Let $D_1 = \C_+$ and $D_2 = \C_-$. Define a $3\times 3$-matrix valued solution $M_n(x,t,k)$, $n = 1,2$, of (\ref{Psilax}) for $k \in D_n$ by the Fredholm integral equations
\begin{align}\label{Mndef}
(M_n)_{ij}(x,t, k) = \delta_{ij} + \int_{\gamma_{ij}^n} \left(e^{(x-x')\widehat{\mathcal{L}(k)}} (\mathsf{U}M_n)(x',t,k) \right)_{ij} dx' , \qquad  i,j = 1, 2,3,
\end{align}
where the contours $\gamma^n_{ij}$, $n = 1, 2$, $i,j = 1,2,3$, are defined by
 \begin{align*}
 \gamma_{ij}^n =  \begin{cases}
 (-\infty,x),  \quad &  \re  l_i(k) \leq \re  l_j(k),
	\\
(\infty, x),  & \re  l_i(k) > \re  l_j(k),
\end{cases} \quad \text{for} \quad k \in D_n,
\end{align*}
with $\mathcal{L} = -ik\Lambda = \diag(l_1, l_2, l_3)$, i.e.,
\begin{align*}
\gamma^1=\begin{pmatrix}\gamma_1&\gamma_1&\gamma_2\\\gamma_1&\gamma_1&\gamma_2\\\gamma_1 &\gamma_1 &\gamma_1\end{pmatrix},\quad \gamma^2=\begin{pmatrix}\gamma_1&\gamma_1&\gamma_1\\\gamma_1&\gamma_1&\gamma_1\\\gamma_2 &\gamma_2 &\gamma_1\end{pmatrix}.
\end{align*}
Solving the matrix factorization problem
\begin{equation}\label{sSSnrelations}  
  s(k) = S_n(k)T_n^{-1}(k), \qquad k \in \bar{D}_n,
\end{equation}
together with the relations
\begin{align} \nonumber
& \left(S_n(k)\right)_{ij} = \delta_{ij} \quad \text{if} \quad \gamma_{ij}^n = (-\infty,x),
	\\ \nonumber
& \left(T_n(k)\right)_{ij} = \delta_{ij} \quad \text{if} \quad \gamma_{ij}^n = (\infty,x),
\end{align}
we infer that
\begin{align*}
   M_n(x,t,k) & = \Psi_1(x,t,k) e^{-i(kx-4k^3t)\hat{\Lambda}} S_n(k)	
  \\
&  = \Psi_2(x,t,k) e^{-i(kx-4k^3t)\hat{\Lambda}} T_n(k), \qquad k \in \bar{D}_n, \ n = 1, 2,
\end{align*}
where the spectral functions $S_n(k)$ and $T_n(k)$ are given in terms of the entries of $s(k)$ by
\begin{subequations}\label{SnTnexplicit}
\begin{align}\nonumber
&  S_1(k) = \begin{pmatrix}
 1 & 0 & \frac{s_{13}}{s_{33}} \\
 0 & 1 & \frac{s_{23}}{s_{33}} \\
 0 & 0 & 1 \\
  \end{pmatrix},
&&
  S_2(k) =  \begin{pmatrix}
  1 & 0 & 0 \\
 0 & 1 & 0 \\
 -\frac{s_{13}^*}{s_{33}^*} & -\frac{s_{23}^*}{s_{33}^*} & 1 \\
\end{pmatrix},
\end{align}
and
\begin{align}\nonumber
&  T_1(k) = \begin{pmatrix}
 s_{11}^* & s_{21}^* & 0 \\
 s_{12}^* & s_{22}^* & 0 \\
 s_{13}^* & s_{23}^* & \frac{1}{s_{33}} \\
   \end{pmatrix},
&&
  T_2(k) =  \begin{pmatrix}
 \frac{s_{22}}{s_{33}^*} & -\frac{s_{12}}{s_{33}^*} & s_{31}^*
   \\
 -\frac{s_{21}}{s_{33}^*} & \frac{s_{11}}{s_{33}^*} & s_{32}^*
   \\
 0 & 0 & s_{33}^* \\
\end{pmatrix}.
\end{align}
\end{subequations}
The expression (\ref{mPsidef}) for $m$ is obtained by taking $m = M_1$ for $k \in D_1$ and $m = M_2$ for $k \in D_2$.
\end{remark}

\appendix

\section{Modified Painlev\'e II RH problem}\label{appA}
Let 
\begin{align*}
& P_1 = \{re^{\frac{\pi i}{6}}\, | \, r \geq 0\} \cup \{re^{\frac{5\pi i}{6}}\, | \, r \geq 0\},
\quad P_2 = \{re^{-\frac{\pi i}{6}}\, | \, r \geq 0\} \cup \{re^{-\frac{5\pi i}{6}}\, | \, r \geq 0\},
\end{align*}
and let $P$ denote the contour $P = P_1 \cup P_2$ oriented as in Figure \ref{P.pdf}.

\begin{figure}
\begin{center}
 \begin{overpic}[width=.4\textwidth]{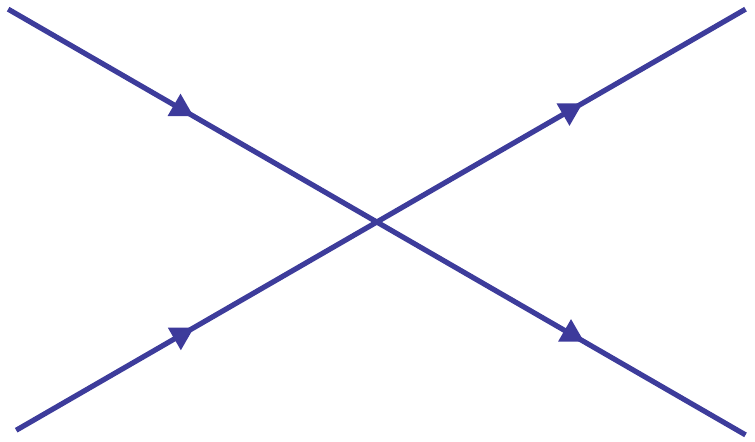}
 \put(72,48){\small $P_1$}
 \put(22,48){\small $P_1$}
 \put(23,7){\small $P_2$}
 \put(72,7){\small $P_2$}
 \put(48.5,23){\small $0$}
 \end{overpic}
   \bigskip
\caption{The contour $P = P_1 \cup P_2$.}
\label{P.pdf}
   \end{center}
\end{figure}
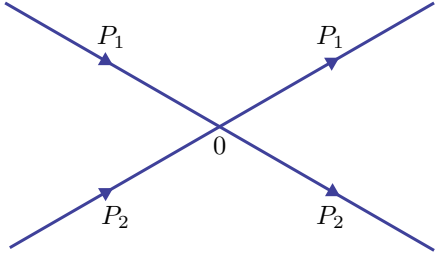

\begin{lemma}[modified Painlev\'e II RH problem]\label{painlevelemma}
Let $s \in \C$ be a complex number and define the matrices $S_1$ and $S_2$ by
$$S_1 = \begin{pmatrix} 1 & 0 & \bar{s} \\ 0 & 1 & s \\ 0 & 0 & 1 \end{pmatrix},
\qquad
S_2 = \begin{pmatrix} 1 & 0 & 0 \\ 0 & 1 & 0 \\ s & \bar{s} & 1  \end{pmatrix}.$$
Then the RH problem 
\begin{itemize}
\item $m^P(y, \cdot)$ is analytic in $\C\setminus P$ with continuous boundary values on $P \setminus \{0\}$;
  \item $m^P_+ = m^P_- v^P$ for $z \in P \setminus \{0\}$;
  \item $m^P =I + O(z^{-1})$ as $z\to\infty$;
  \item $m^P = O(1)$ as $z \to 0$;
\end{itemize}
where
$$v^P(y,z) = e^{-i(yz - \frac{4z^3}{3}) \hat{\Lambda}} S_n, \qquad z \in P_n, \ n = 1, 2,$$
has a unique solution $m^P(y, z)$ for each $y \in \R$. 
Moreover, there are smooth functions $\{m_j^P(y)\}_1^\infty$ of $y \in \R$ with decay as $y \to -\infty$ such that, for each integer $N \geq 0$,
\begin{align}\label{mPasymptotics}
  m^P(y, z) = I + \sum_{j=1}^N \frac{m_j^P(y)}{z^j} + O(z^{-N-1}), \qquad z \to \infty,
\end{align}
uniformly for $y$ in compact subsets of $\R$ and for $\arg z \in [0,2\pi]$. The (13)-entry of the leading coefficient $m_1^P$ is given by
$$(m_1^P(y))_{13} = \frac{u_P(y)}{2\sqrt{2}},$$
where $u_P(y) \equiv u_P(y; s)$ satisfies the modified Painlev\'e II equation (\ref{complexpainleveII}) and has constant phase, that is, $\arg u_P$ is independent of $y$. 
\end{lemma}
\begin{proof}
The jump matrix $v^P$ obeys the symmetries
\begin{align}\label{vPsymm}
v^P(y, z) = (v^P)^{\dagger}(y, \bar{z}) =\mathcal{A}\overline{v^P(y, -\bar{z})}\mathcal{A}.
\end{align}
We infer from the first of these symmetries that the RH problem for $m^P$ admits a vanishing lemma, see \cite[Theorem 9.3]{z-1989}. As in Section \ref{th1proofsec}, this implies that there exists a unique solution $m^P$ which admits an expansion of the form (\ref{mPasymptotics}). A Deift-Zhou steepest descent analysis shows that the coefficients $m_j^P$ (and their $y$-derivatives) have exponential decay as $y \to -\infty$.

Let $\phi(y,z)=m^P(y,z)e^{-i(yz-\frac{4 z^3}{3})\Lambda}$. 
Then the function $\mathcal{U}(y,z)$ defined by
\begin{align}\label{U}
\mathcal{U} \doteq \phi_y\phi^{-1} = \big(m_y^P - iz m^P \Lambda\big)(m^P)^{-1}
\end{align}
is an entire function of $z$; hence $\mathcal{U}(y,z) = \mathcal{U}_0(y) +\mathcal{U}_1(y) z$.
Equation \eqref{U} then becomes
\begin{align}\label{U1}
 m_y^P - iz m^P \Lambda = (\mathcal{U}_0+\mathcal{U}_1 z)m^P
\end{align}
Substituting the expansion (\ref{mPasymptotics}) into \eqref{U1}, we find
\begin{align*}
\mathcal{U}_1 = -i \Lambda,\quad \mathcal{U}_0 =i [\Lambda, m_1^P].
\end{align*}
Similarly,
\begin{align}\label{A}
\mathcal{V} \doteq \phi_z \phi^{-1}=\bigg(m_z^P-i(y-4 z^2)m^P  \Lambda\bigg)(m^P)^{-1}
\end{align}
is entire, and hence $\mathcal{V}=\mathcal{V}_0+\mathcal{V}_1z+\mathcal{V}_2 z^2$.
Substituting \eqref{mPasymptotics} into \eqref{A}, we find
\begin{align*}
&\mathcal{V}_2=4i\Lambda, \quad \mathcal{V}_1=-4i [\Lambda, m_1^P], \quad \mathcal{V}_0=-\mathcal{V}_1 m_1^P-4i[\Lambda, m_2^P]- iy\Lambda
\end{align*}
Substituting \eqref{mPasymptotics} into (\ref{U1}), it follows that
\begin{align*}
  m^P_{1y}+i [\Lambda,m^P_2]=\mathcal{U}_0m^P_1,
\end{align*}
which gives
\begin{align*}
\mathcal{V}_0=-\mathcal{V}_1 m_1^P - 4(\mathcal{U}_0m^P_1- m^P_{1y})- iy\Lambda
\end{align*}
We have shown that $\phi$ obeys the Lax pair equations
\begin{align}\label{mPlax}
\begin{cases}
  \phi_y = \mathcal{U}\phi, \\
  \phi_z = \mathcal{V} \phi,
\end{cases}
 \qquad y \in \R^2, \ z \in \C \setminus P,
\end{align}
where $\mathcal{U}$  and $\mathcal{V}$ are expressed in terms of $m_1^P(y)$. 

As a consequence of (\ref{vPsymm}), $m^P$ obeys the symmetries
\begin{align}\label{mPsymm}
 m^P(y, z) = (m^P)^{\dagger}(y, \bar{z})^{-1} =\mathcal{A}\overline{m^P(y, -\bar{z})}\mathcal{A}.
\end{align}
In particular, the leading coefficient $m_1^P$ satisfies
$$m_1^P(y) = - m_1^P(y)^\dagger = -\mathcal{A}\overline{m_1^P(y)}\mathcal{A}.$$
Hence we can write
$$m_1^P(y) = \begin{pmatrix}\psi_1(y)&\psi_2(y)&\psi_3(y)\\-\overline{ \psi_2(y)}&\psi_1(y)&-\overline{ \psi_3(y)}\\-\overline{ \psi_3(y)}&\psi_3(y)&\psi_4(y)\end{pmatrix},$$
where $\{\psi_j(y)\}_1^4$ are complex-valued function such that $\psi_1(y), \psi_4(y) \in i \R$.
The compatibility condition
\begin{align*}
\mathcal{U}_z - \mathcal{V}_y+\mathcal{U}\mathcal{V}-\mathcal{V}\mathcal{U}=0
\end{align*}
of the Lax pair (\ref{mPlax}) is then equivalent to the following four equations:
\begin{subequations}
\begin{align} \label{compa}
\psi_1'' + 2i (\psi_3 \bar{\psi}_3 )'&=0,
	\\ \label{compb}
\psi_2''-4i \psi_3 \psi_3' &=0,
	\\ \label{compc}
\psi_3'' -2i \bar \psi_3 \psi_2' + \psi_3(y + 2i \psi_1' - 2i\psi_4')&=0,
	\\ \label{compd}
\psi_4'' - 4i (\psi_3 \bar{\psi}_3)'&=0.
\end{align}
\end{subequations}
Since $m_1^P(y)$ and its derivatives decay as $y \to -\infty$, equations \eqref{compa}, \eqref{compb}, and \eqref{compd} yield
\begin{align}\label{psi124}
\psi'_1=-2i|\psi_3|^2, \quad \psi'_2 = 2i \psi_3^2, \quad \psi'_4=4i|\psi_3|^2.
\end{align}
Substituting \eqref{psi124} into \eqref{compc}, we find
\begin{align}\label{psi3eq}
\psi_3'' + y\psi_3+16 \psi_3 |\psi_3|^2=0.
\end{align}
Writing $\psi_3(y) = r(y) e^{i\alpha(y)}$ with $r(y), \alpha(y) \in \R$, (\ref{psi3eq}) reduces to the pair of equations
\begin{subequations}
\begin{align}\label{req}
r'' + 16 r^3 + yr - (\alpha')^2r = 0, 
	\\ \label{alphaeq}
2r' \alpha' + r \alpha'' = 0.
\end{align}
\end{subequations}
Equation (\ref{alphaeq}) yields $r^2 \alpha' = c_0$, where $c_0 \in \R$ is a constant. Using this relation to eliminate $\alpha'$ from (\ref{req}), we obtain
$$r'' + 16 r^3 + yr - c_0^2 r^{-3} = 0.$$
The decay of $\psi_3$ and its derivatives as $y \to -\infty$ shows that we must have $c_0 = 0$. 
Hence $\alpha(y) = \arg \psi_3(y)$ is independent of $y$. 
The lemma follows by setting $u_P(y) \doteq 2\sqrt{2} \psi_3(y)$.
\end{proof}

\section{Model problem for Sector $\mathcal{P}_\geq$}\label{appB}
Given $z_0 \geq 0$, let 
\begin{align} \nonumber
&Z_1 = \bigl\{z_0+ re^{\frac{i\pi}{6}}\, \big| \, 0 \leq r < \infty\bigr\}, && Z_2 = \bigl\{-z_0 + re^{\frac{5i\pi}{6}}\, \big| \, 0 \leq r < \infty\bigr\},  
	\\ \nonumber
&Z_3 = \bigl\{-z_0 + re^{-\frac{5i\pi}{6}}\, \big| \, 0 \leq r < \infty\bigr\}, && Z_4 = \bigl\{z_0 + re^{-\frac{i\pi}{6}}\, \big| \, 0 \leq r < \infty\bigr\}, 
	\\ \label{ZdefIVg}
& Z_5 = \bigl\{r\, \big|  -z_0 \leq r \leq z_0\bigr\},
\end{align}
and let $Z \equiv Z(z_0)$ denote the contour $Z = \cup_{j=1}^5 Z_j$ oriented as in Figure \ref{Z.pdf}. 
Suppose
\begin{subequations}\label{psumIV}
\begin{align}
  p_1(t,z) = s + \sum_{j=1}^n \frac{p_{1,j}z^j}{t^{j/3}},
\end{align}
is a polynomial in $z t^{-1/3}$ with coefficients $s \in \C$ and $\{p_{1,j}\}_1^n \subset \C$ for some integer $n \geq 0$. Define the row-vector valued function $p(t,z)$ by 
\begin{align}
p(t,z) = \begin{pmatrix} p_1(t,z) & p_2(t,z)\end{pmatrix}, \qquad p_2(t,z) \doteq \overline{p_1(t,-\bar{z})}.
\end{align}
\end{subequations}

The long-time asymptotics in $\mathcal{P}_\geq$ is related to the solution $m^Z$ of the following family of RH problems parametrized by $y \geq 0$, $t \geq 0$, and $z_0 \geq 0$:
\begin{itemize}
\item $m^Z(y, t, z_0, \cdot)$ is analytic in $\C\setminus Z$ with continuous boundary values on $Z \setminus \{\pm z_0\}$;
  \item $m^Z_+ = m^Z_- v^Z$ for $z \in Z \setminus \{\pm z_0\}$;
  \item $m^Z =I + O(z^{-1})$ as $z\to\infty$;
  \item $m^Z = O(1)$ as $z \to \pm z_0$;
\end{itemize}
where the jump matrix $v^Z(y, t, z_0, z)$ is defined by
\begin{align}\label{vZdefIVg}
v^Z(y, t, z_0, z) = 
\begin{cases}
\begin{pmatrix} \I_{2\times 2} & p^{\dag}(t, \bar{z})e^{-2i(yz-\frac{4 z^3}{3})} \\ \0_{1\times 2} & 1\end{pmatrix}, &  z \in Z_1 \cup Z_2, 
	\\
\begin{pmatrix} \I_{2\times 2} & \0_{2\times 1}\\ p(t,z)e^{2i(yz-\frac{4 z^3}{3})}  & 1\end{pmatrix}, &  z \in Z_3 \cup Z_4, 
  	\\
\begin{pmatrix}\I_{2\times 2}& p^{\dag}(t, \bar{z})e^{-2i(yz-\frac{4 z^3}{3})} \\ p(t,z)e^{2i(yz-\frac{4 z^3}{3})} &1+ p(t,z) p^{\dag}(t,\bar{z}) \end{pmatrix}, &  z \in Z_5,
\end{cases}
\end{align}
with $p(t,z)$ given by (\ref{psumIV}). 
Define the parameter subset $\mathbb{P} \subset \R^3$ by
\begin{align}\label{parametersetIVg}
\mathbb{P} = \{(y,t,z_0) \in \R^3 \, | \, 0 \leq y \leq C_1, \, t \geq 1, \, \sqrt{y}/2 \leq z_0 \leq C_2\},
\end{align}
where $C_1,C_2 > 0$ are constants.

\begin{figure}
\begin{center}
 \begin{overpic}[width=.5\textwidth]{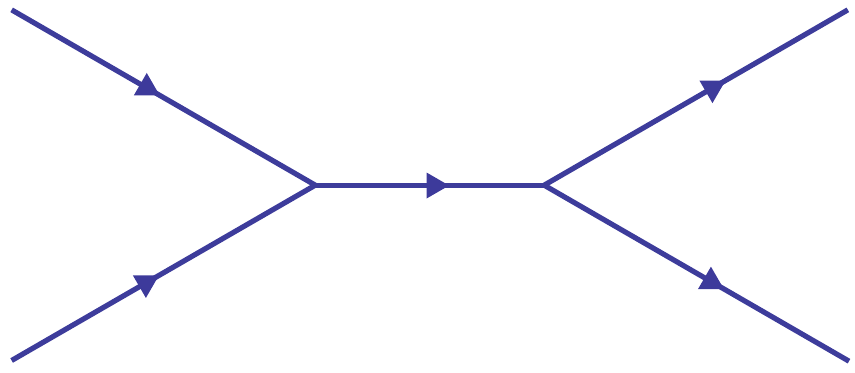}
 \put(78,36){\small $Z_1$}
 \put(17,36){\small $Z_2$}
 \put(17,5){\small $Z_3$}
 \put(78,5){\small $Z_4$}
 \put(48,25){\small $Z_5$}
 \put(62,17){\small $z_0$}
 \put(33,17){\small $-z_0$}
 \end{overpic}
 \caption{The contour $Z$.}
 \label{Z.pdf}
   \end{center}
\end{figure}
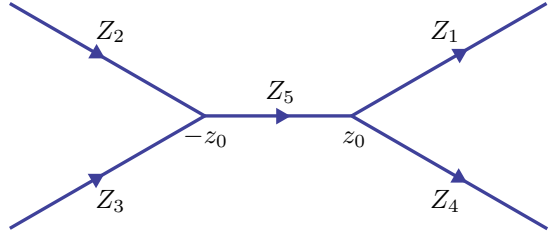

\begin{lemma}[Model problem for Sector $\mathcal{P}_\geq$]\label{ZlemmaIVg}
Let $p(t,z)$ be of the form (\ref{psumIV}) for some $s \in \C$ and $\{p_{1,j}\}_1^n \subset \C$.

\begin{enumerate}[$(a)$]
\item The RH problem for $m^Z$ with jump matrix $v^Z$ given by (\ref{vZdefIVg}) has a unique solution $m^Z(y, t, z_0, z)$ whenever $(y,t,z_0) \in \mathbb{P}$. 

\item There are smooth functions $\{m_{jl}^Z(y)\}$ such that, for each integer $N \geq 1$,
\begin{align}\label{mYasymptoticsIVg}
&  m^Z(y, t, z_0, z) = I + \sum_{j=1}^N \sum_{l=0}^N \frac{m_{jl}^Z(y)}{z^j t^{l/3}}  + O\biggl(\frac{t^{-(N+1)/3}}{|z|} + \frac{t^{-1/3}}{|z|^{N+1}}\biggr), \quad z \to \infty,
\end{align}
uniformly with respect to $\arg z \in [0, 2\pi]$ and $(y,t,z_0) \in \mathbb{P}$. 

\item $m^Z(y, t, z_0, z)$ is uniformly bounded for $z \in \C\setminus Z$ and $(y,t,z_0) \in \mathbb{P}$.

\item $m^Z$ obeys the symmetries
\begin{align}\label{mZsymm}
 m^Z(y, t, z_0, z) = (m^Z)^{\dagger}(y, t, z_0, \bar{z})^{-1}, \quad m^Z(y, t, z_0, z)=\mathcal{A}\overline{m^Z(y, t, z_0, -\bar{z})}\mathcal{A}.
\end{align}

\item The (13)-entry of the leading coefficient $m_{10}^Z$ is given by
$$(m_{10}^Z(y))_{13} = \frac{u_P(y;s)}{2\sqrt{2}},$$
where $u_P(y; s)$ is the smooth solution of the modified Painlev\'e II equation (\ref{complexpainleveII}) associated with $s$ according to Lemma \ref{painlevelemma}.

\end{enumerate}
\end{lemma}
\begin{proof}
We have
\begin{align*}
\re\bigg(-2i\bigg(yz - \frac{4z^3}{3}\bigg)\bigg)
= -\frac{8 r^3}{3}-4 \sqrt{3} r^2 z_0 + r \left(y-4 z_0^2\right)
\leq  -\frac{8 r^3}{3}-4 \sqrt{3} r^2 z_0,
\end{align*}
for all $z = z_0 + re^{\frac{\pi i}{6}} \in Z_1$ with $r \geq 0$,  $z_0 \geq 0$, and $0 \leq y \leq 4z_0^2$. Consequently, 
\begin{align*}
|e^{2i(y z + \frac{4z^3}{3})}| \leq Ce^{-|z-z_0|^2(z_0 + |z-z_0|)}, \qquad  z \in Z_1,
\end{align*}
uniformly for $(y,t,z_0) \in \mathbb{P}$. Analogous estimates hold for $z \in Z_j$, $j = 2,3,4$, and $|e^{\pm 2i(y z - 4z^3/3)}| = 1$ for $z \in Z_5$, showing that $v^Z - I$ has uniform decay for large $z$. 

The jump matrix $v^Z$ obeys the same symmetries (\ref{vPsymm}) as $v^P$. In particular, $v^Z$ is Hermitian and positive definite on $Z \cap \R$ and satisfies $v^Z(y,t,z_0,z) = (v^Z)^\dagger(y,t,z_0,\bar{z})$ on $Z \setminus \R$. This implies the existence of a vanishing lemma \cite{z-1989} from which we deduce the unique existence of the solution $m^Z$. The symmetries (\ref{mZsymm}) follow from the symmetries of $v^Z$.

Let $m^P(y,z) \equiv m^P(y,z;s)$ solve the same RH problem as $m^Z$ except that the polynomial $p(t,z)$ in the jump matrix (\ref{vZdefIVg}) is replaced with its leading term $s$. 
Then (up to a trivial contour deformation) $m^P$ is the solution of Lemma \ref{painlevelemma} corresponding to $s$. The remainder of the proof is analogous to the corresponding proof for the mKdV equation (see \cite{CL2019}) and consists of considering the RH problem satisfied by the quotient $m^Y (m^P)^{-1}$.
\end{proof}

\noindent
{\bf Acknowledgement}  {\it
The authors acknowledge support from the G\"oran Gustafsson Foundation, the European Research Council, Grant Agreement No. 682537, the Swedish Research Council, Grant No. 2015-05430, and the National Science Foundation of China, Grant No. 11671095.}

\end{document}